\documentclass{amsart}
\usepackage{amssymb}
\usepackage[utf8]{inputenc}
\newtheorem{thm}{Theorem}
\newtheorem{lem}[thm]{Lemma}
\newtheorem{defi}[thm]{Definition}
\newtheorem{prop}[thm]{Proposition}
\newtheorem{coro}[thm]{Corollary}
\newtheorem{conj}{Conjecture}

\title{The Mesyan conjecture: a restatement and a correction}
\author[P. S. Fagundes]{Pedro Souza Fagundes}
\address{IMECC, Universidade Estadual de
		Campinas, Rua S\'ergio Buarque de Holanda, 651, Cidade Universit\'aria ``Zeferino Vaz'', Distr. Bar\~ao Geraldo, Campinas, S\~ao Paulo, Brazil, CEP
		13083-859}\email{pedro.fagundes@ime.unicamp.br}
\author[T.C. de Mello]{Thiago Castilho de Mello}
\address{Instituto de Ci\^encia e Tecnologia, Universidade Federal de S\~ao Paulo, Av. Cesare M. Giulio Lattes, 1201, 12.247-014,	S\~ao Jos\'e dos Campos, SP, Brazil}\email{tcmello@unifesp.br}
\author[P. H. S. dos Santos]{Pedro Henrique da Silva dos Santos}
\address{Instituto de Ci\^encia e Tecnologia, Universidade Federal de S\~ao Paulo, Av. Cesare M. Giulio Lattes, 1201, 12.247-014,	S\~ao Jos\'e dos Campos, SP, Brazil}\email{silva.pedro@unifesp.br}

\dedicatory{Dedicated to Professor Vesselin Drensky on the occasion of his 70th birthday.}

\begin{document}

\maketitle

\begin{abstract}
    The well-known Lvov-Kaplansky conjecture states that the image of a multilinear polynomial $f$ evaluated on $n\times n$ matrices is a vector space. A weaker version of this conjecture, known as the Mesyan conjecture, states that if $m=\deg f$ and $n\geq m-1$ then its image contains the set of trace zero matrices. Such conjecture has been proved for polynomials of degree $m \leq 4$. The proof of the case $m=4$ contains an error in one of the lemmas. In this paper, we correct the proof of such lemma and present some evidences which allow us to state the Mesyan conjecture for the new bound $n \geq \frac{m+1}{2}$, which cannot be improved.
  
\end{abstract}

\section{Introduction}

Let $K$ be a field and let $M_n(K)$ denote the algebra of $n\times n$ matrices over $K$. A famous problem known as Lvov-Kaplansky conjecture asserts: the image of a multilinear polynomial (in noncommutative variables) on $M_n(K)$ is a vector space. Such conjecture is equivalent to the following: the image of a multilinear polynomial on $M_n(K)$ is $\{0\}$, $K$ (viewed as the set of scalar matrices), $sl_n(K)$ (the set of traceless matrices) or $M_n(K)$.

Although proving that some subset is a vector space seems to be in a first look a simple problem, a solution to Lvov-Kaplansky conjecture is known only for $n=2$ \cite{K-BMR, Malev2}. The case $n=3$ has interesting progress, but not a solution \cite{K-BMR3}. 
This conjecture motivated other studies related to images of polynomials. For instance, papers on images of polynomials on some subalgebras of $M_n(K)$, images of Lie, and Jordan polynomials on Lie and Jordan algebras have been published since then (see \cite{Lie, Fagundes, FdeM, GdM,  K-BMRLie,MalevQ,MalevJ}). For a nice compilation of results on images of polynomials, we recommend the survey \cite{Survey}.

An analogous of the Lvov-Kaplansky conjecture for the infinite dimensional case, i.e., for the algebra $A=\text{End}(V)$, where $V$ is a countably infinite-dimensional vector space over $K$ was studied in \cite{Vitas1}. In such paper, the author proved that if $f$ is any nonzero multilinear polynomial, than the image of $f$ is $A$.

A weakening of the Lvov-Kaplansky conjecture is the so called Mesyan conjecture \cite[Conjecture 11]{Mesyan}:

\begin{conj}\label{conj0}
    Let $K$ be a field, $n\geq 2$ and $m\geq 1$ be integers and $f(x_1, \dots, x_m)$ a nonzero multilinear polynomial in $K \langle x_1, \dots, x_m \rangle$.  If $n\geq m-1$, then the image of $f$ contains all trace zero matrices.
\end{conj}

The above conjecture is based on the following result (see \cite[Proposition 10]{Mesyan}).

\begin{prop}\label{propo}
    Let $K$ be a field, $n\geq 2$ and $m\geq 1$ be integers, and $f(x_1, \dots, x_m)$ be a nonzero multilinear polynomial in $K  \langle x_1, \dots, x_m \rangle$. If $n\geq m - 1$, then the $K$-subspace of $M_n(K)$ generated by the image of $f$ contains $sl_n(K)$.
\end{prop}

In fact, once one assumes the Lvov-Kaplansky conjecture is true, the sentence ``{the image of $f$ on $M_n(K)$ contains $sl_n(K)$}'' is equivalent to ``$f$ is not an identity nor a central polynomial of $M_n(K)$''. On the other hand, it is well-known that $M_n(K)$ has no identities or central polynomials of degree $m\leq n+1$, and this makes Conjecture \ref{conj0} a particular case of the Lvov-Kaplansky conjecture. In particular, a counter-example to Mesyan conjecture is a counter-example for the Lvov-Kaplanksy conjecture.

In this paper, we present a more general result than Proposition \ref{propo}, and restate the conjecture for a more general case (see Conjecture \ref{conjN} below). We also discuss the relation of Conjecture \ref{conjN} and minimal degrees of central polynomials and identities for $M_n(K)$.

Positive solutions for Conjecture \ref{conj0} have been presented for the algebra $M_\infty(K)$, of finitary matrices and for $m\leq 4$. In \cite{Vitas2} the author proved an analogue of the Mesyan conjecture for algebra $M_\infty(K)$, namely, if $K$ is an infinite field and $f$ is a nonzero multilinear polynomial, then the image of $f$ on $M_\infty(K)$ contains $sl_\infty(K)$ (the set of trace zero finitary matrices). 

The case $m=2$ is a direct consequence of results of Shoda \cite{Shoda} (for the characteristic zero case) and by Albert and Muckenhoupt \cite{AM} (for the positive characteristic case) where they prove that any trace zero matrix is given by a commutator of two matrices, while the case $m=3$ was proved by Mesyan himself in \cite{Mesyan}.

The case $m=4$ was presented by Buzinsky and Winstanley in \cite{BW}, but their proof contains a crucial error in one of its lemmas, so the solution is not correct. In this paper, we present a correction for the such lemma, confirming the positive solution of Mesyan conjecture for $m=4$.

\section{Preliminaries}

In this section we define the basic objects and present the basic results necessary to the paper. We start with the definition of a multilinear polynomial. Throughout the paper, unless otherwise stated, $K$ will denote an arbitrary field and all algebras are considered over $K$.

\begin{defi}
    Let $m$ be a positive integer. By $K\langle x_1, \dots, x_m\rangle$ we denote the free associative algebra, freely generated by $\{x_1, \dots, x_m\}$.   The elements of $K\langle x_1, \dots, x_m\rangle$ will be called \emph{polynomials} in the noncommutative variables $x_1, \dots, x_m$. A polynomial $f(x_1, \dots, x_m)$ is said to be \emph{multilinear} if it can be written as
    \[f(x_1,\dots,x_m)=\sum_{\sigma\in S_m}\alpha_{\sigma}x_{\sigma(1)} \cdots x_{\sigma(m)},\]
    where $S_m$ denotes the group of permutations of $\{1,\dots, m\}$ and $\alpha_\sigma \in K$, for $\sigma\in S_m$.     
\end{defi}

For a given $K$-algebra $A$, a polynomial $f(x_1,\dots,x_m)\in K\langle  X_1, \dots, X_m\rangle$ defines a map (also denoted by $f$)
\[\begin{array}{cccc}
		f: & A^m & \longrightarrow & A  \\
		& (a_1,\dots,a_m) & \mapsto & f(a_1,\dots,a_m) \\
	\end{array} \]
The image of such map $f$ is called \emph{the image of the polynomial $f$ on $A$} and will be denoted by $f(A)$.

Some well-known properties of the set $f(A)$ are given below.

\begin{prop}\label{properties}
Let $A$ be an algebra and $f(x_1,\dots,x_m)=\sum_{\sigma\in S_n}\alpha_{\sigma}x_{\sigma(1)}\cdots x_{\sigma(m)}$ be a multilinear polynomial. Then
\begin{enumerate}
    \item $f(A)$ is closed under automorphisms of $A$. In particular, $f(A)$ is closed under conjugation by invertible elements.
    \item $f(A)$ is closed under scalar multiplication.
    \item The linear span of $f(A)$ is a Lie ideal of $A$. 
    \item If $\sum_{\sigma\in S_m}\alpha_{\sigma}\neq 0$ then $f(A)=A$.
\end{enumerate}
\end{prop}

The theory of images of polynomials on algebras has strong connections with the theory of polynomial identities (PI-theory). For instance, a polynomial identity for an algebra $A$ is a polynomial whose image is $\{0\}$ and a central polynomial is a polynomial whose image lies in $Z(A)$, the center of the algebra $A$.
The set of all polynomial identities of an algebra $A$ is an ideal of $K\langle x_1,x_2, \dots\rangle$ which is invariant under endomorphisms of $K\langle x_1,x_2, \dots\rangle$. It is denoted by $T(A)$.

Some techniques from PI-theory are useful in studying images of polynomials on algebras. For instance, when $A=M_n(K)$, the $m$-generated algebra of generic matrices is known to be isomorphic to the quotient algebra \[F_m(M_n(K))=\displaystyle \frac{K\langle x_1, \dots, x_m \rangle}{T(M_n(K))\cap K\langle x_1, \dots, x_m \rangle},\] and working module $T(M_n(K))$ is equivalent to work in the algebra of generic matrices, see for instance \cite[Chapter 7]{Drensky}.

Let us denote by $St_{k}$ the standard polynomial of degre $k$:
\[St_{k}(x_1,\dots, x_{k})=\sum_{\sigma\in S_{k}}(-1)^\sigma x_{\sigma(1)} \cdots x_{\sigma(k)}.\]
The following is a well-known fact about identities in matrices

\begin{thm}[Amitsur-Levitszky]
    The algebra $M_n(K)$ satisfies the polynomial identity $St_{2n}$. Moreover, $M_n(K)$ does not satisfy any polynomial identity of degree less than $2n$ and any polynomial identity of degree $2n$ of $M_n(K)$ is a scalar multiple of $St_{2n}$.
\end{thm}

\begin{coro}\label{central}
    The algebra $M_n(K)$ does not have central polynomials of degree less than $2n$.
\end{coro}

\begin{proof}
Assume $f(x_1,\dots,x_m)$ is a multilinear central polynomial for $M_n(K)$. Then the commutator $g=[f(x_1,\dots,x_m),x_{m+1}]$ is a polynomial identity for $M_n(K)$. As a consequence $m+1\geq 2n$, which means $m\geq 2n-1$. If $m=2n-1$ then $g$ is a polynomial identity of degree $2n$, and it must be a scalar multiple of $St_{2n}$, but writing $g$ as a sum of nonzero monomials gives us at most $2(2n-1)!$ summands, while in $St_{2n}$ we have $(2n)!$ summands. A contradiction. Hence $m\geq 2n$.
\end{proof}

\section{A new bound for the Mesyan conjecture}

The main goal of this section is to present evidences which will allow us to state the Mesyan conjecture in a more general setting. 

Let us assume for a moment that the Lvov-Kaplansky conjecture is true. If $f$ is a polynomial of degree $m \leq 2n-1 $ then by Corollary \ref{central}, $f$ is not a central polynomial nor an identity for $M_n(K)$ and by our assumption  $f(M_n(K))$ is $sl_n(K)$ or $M_n(K)$ which, in both cases, contains $sl_n(K)$.

The above fact suggest that the Mesyan conjecture should be stated in a more general setting, namely for $n\geq \frac{m+1}{2}$. Also, this bound cannot be improved once $St_{2n}$ is a polynomial identity of degree $2n$ for $M_n(K)$.

We now present one more evidence that the conjecture should be stated in this setting. We will prove a more general version of Proposition \ref{propo}:

        \begin{thm}\label{newbound}
            Let $K$ be a field, $n \geq 2$ and $m \geq 2$ positive integers such that $char(K)$ does not divide $n$ and let $f(x_1, \cdots , x_m)$ be a non-zero multilinear polynomial in $K \langle x_1, x_2, \ldots, x_m \rangle$. If $n \geq \frac{m+ 1}{2}$ , then the $K$-subspace span$(f(M_n(K))$ contains $sl_n(K)$. 
        \end{thm}

Recall that for polynomials of degree $m=2$, the Lvov-Kaplanksy conjecture is a consequence of Proposition \ref{properties} (4) and results of Shoda and Albert and Muckenhoupt. Also, for $m\geq 3$ we have $\frac{m+1}{2}\leq m-1$, which shows that the above is a generalization of Proposition \ref{propo}.

Before proving the above theorem, we must first recall the following technical lemma from Amitsur and Rowen (\cite{AR}, Proposition 1.8).

\begin{lem} \label{lemasimiliar}
Let $D$ be a division ring, $n \geq 2$ an integer, and $A \in M_n(D)$ noncentral 
matrix. Then, $A$ is similar to a matrix in $M_n(D)$ with at most one non-zero entry on
the main diagonal. In particular, if $A$ has trace zero, then it is similar to a matrix in $M_n(D)$
with only zeros on the main diagonal.
\end{lem}

	We are now ready to present a proof of Theorem  \ref{newbound}.

\begin{proof}(of Theorem \ref{newbound})
Let \[ f(x_1, \cdots , x_m) = \displaystyle\sum_{\sigma \in S_m} \alpha_\sigma x_{\sigma(1)} \cdots x_{\sigma(m)}.\]

We first recall that if $m=2$ and $f$ is a nonzero polynomial, the image of $f$ is $sl_n(K)$ or $M_n(K)$. Therefore, we can assume $m \geq 3$. 

Without loss of generality we may assume that $\alpha_{(1)} \neq 0$. 
Let  $i, j \; \in \{1, \dots , n \}$ such that $i\neq j$. We will consider two cases: when $m$ is even and when $m$ is odd.
 
 \begin{enumerate}
 \item[\textbf{Case 1:}] $m$ is even. 
 
 Let $m = 2k$, with $k$ integer and $k \geq 2$, then $n \geq k + \frac{1}{2 }$. Since $k$ and $n$ are  integers, then $n \geq k + 1$, therefore $n - 2 \geq k - 1$. So let $l_1,\dots , l_{k-1}$ be  $k - 1$ distinct elements in $ \{1, \dots ,n \} - \{i,j \}$. Then we have
\[f(e_{ii}, e_{ij}, e_{jl_1}, e_{l_1l_1}, e_{l_1l_2}, e_{l_2l_2}, \dots , e_{l_{k-2}l_{k-1}}, e_{l_{k-1}j}) = \alpha_{(1)} e_{ij}.\]

Since $\alpha_{(1)} \neq 0$ and $e_{ij}$ with $i, j$ distinct is a matrix whose diagonal contains only zeros, we have $span(f(M_n(K))$ contains all matrices with zeros on the main diagonal.

Let now $A\in M_n(K)$ be a nonzero trace zero matrix, we must show that $A \in span(f(M_n(K))$. Since char$(K)$ does not divide $n$ then $A$ is a non-central matrix,  hence by Lemma \ref{lemasimiliar} there exists $B \in span(f(M_n(K))$ such that $A = PBP^{-1}$, for some invertible matrix $P\in M_n(K)$ and by Proposition \ref{properties} (1) we conclude that $A \in span(f(M_n(K))$. 

\item[\textbf{Case 2:}] $m$ is odd.

Let $m = 2k -1$, where $k$ is an integer, such that $k \geq 2$. Since $n \geq \frac{m +1}{2}$ then $n-2 \geq k-2$. Therefore, we can find  $k- 2$ distinct elements in $ \{1, \cdots , n \} - \{ i, j\}$, which we will denote by $l_1, \cdots , l_{k-2}$. Hence,
$$f(e_{ii}, e_{ij}, e_{jl_1}, e_{l_1l_1}, e_{l_1l_2}, e_{l_2l_2} \cdots , e_{l_{k-3}l_{k-2}}, e_{l_{k-2}l_{k-2}}, e_{l_{k-2}j}) = \alpha_{(1)} e_{ij}.   $$

As in the previous case, if $A \in sl_n(K)$ then $A \in span(f(M_n(K))$.
 \end{enumerate}
\end{proof}

Now we restate Mesyan conjecture in light of Theorem \ref{newbound} and of the discussion of the beginning of this section.

\begin{conj}[Mesyan conjecture restated]\label{conjN} 
Let $K$ be a field, $n \geq  2$ and $m \geq 1$ be integers, and let $f(x_1, \ldots , x_m)$ be a
non-zero multilinear polynomial in $K \langle x_1, \ldots , x_m \rangle$. If $m \leq  2n -1$, then $sl_n (K) \subseteq f(M_n(K))$. 
\end{conj} 

\section{The Mesyan conjecture for polynomials of degree 4}

In this last section our main goal is to discuss the following result given in \cite{BW} and also to give a correction of a particular lemma used in its proof.

\begin{thm}\label{t1}
Let $n\geq 3$ and let $K$ be an algebraically closed field of characteristic zero. Then the image of a nonzero multilinear polynomial $f(x_{1},x_{2},x_{3},x_{4})$ on the matrix algebra $M_{n}(K)$ contains $sl_{n}(K)$. 
\end{thm}

For the sake of completeness we will present the proof of Theorem \ref{t1} in next. However some preliminaries lemmas will be required and we will present them in the following without their proofs. The first one is given in \cite[Lemma 6]{BW}.

\begin{lem}\label{l4}
Let $n\geq 3$ be an integer, let $K$ be a field of characteristic zero and let $a_{i,j}\in K$ such that $\sum_{i=1}^{n}a_{i,i}=0$. Then there exist $A,B,C\in M_{n}(K)$ such that $[A,[[A,B],[A,C]]]=\sum_{i=1}^{n}a_{i,i}e_{i,i}+\sum_{i=1}^{n-1}a_{i,i+1}e_{i,i+1}$.
\end{lem}

For the next lemma see \cite[Lemma 1.2]{AR}.

\begin{lem}\label{l5}
Let $K$ be any field and let $A\in M_{n}(K)$ be a diagonal matrix with pairwise different entries in the main diagonal. Then $[A,M_{n}(K)]$ is the set of matrices whose diagonals entries are all $0$.
\end{lem}

\begin{proof}[Proof of Theorem \ref{t1}]
We start the proof by reducing the polynomial $f$ to a proper one. This can be done by considering the degree three multilinear polynomials obtained from $f$ through the evaluation of some variable $x_{i}, i=1,\dots,4$, by $1$. In case one of these four polynomials is nonzero then we are able to use the Mesyan's result (see \cite[Theorem 13]{Mesyan}) to obtain the desired conclusion. Otherwise, since $char(K)=0$ then we have $f$ as a proper polynomial (see for instance \cite[Exercise 4.3.6]{Drensky}). 

Hence we may write $f$ as 
\begin{eqnarray}\nonumber
f(x_{1},x_{2},x_{3},x_{4})=L(x_{1},x_{2},x_{3},x_{4})+\alpha_{1}[x_{1},x_{2}][x_{3},x_{4}]+\alpha_{2}[x_{1},x_{3}][x_{2},x_{4}]\\\nonumber 
+\alpha_{3}[x_{1},x_{4}][x_{2},x_{3}]
+\alpha_{4}[x_{2},x_{3}][x_{1},x_{4}]
+\alpha_{5}[x_{2},x_{4}][x_{1},x_{3}]+\alpha_{6}[x_{3},x_{4}][x_{1},x_{2}]
\end{eqnarray}
where $\alpha_{1},\dots,\alpha_{6}\in K$.
Using a Hall basis for the multilinear Lie polynomials of degree four (see \cite[section 2.3]{Bahturin}), we may write the Lie polynomial $L$ as 
\begin{align*}
L(x_{1},x_{2},x_{3},x_{4})= & \beta_{1}[[[x_{2},x_{1}],x_{3}],x_{4}]+\beta_{2}[[[x_{3},x_{1}],x_{2}],x_{4}]+\beta_{3}[[[x_{4},x_{1}],x_{2}],x_{3}]\\
+ & \beta_{4}[[x_{4},x_{1}],[x_{3},x_{2}]]+\beta_{5}[[x_{4},x_{2}],[x_{3},x_{1}]]+\beta_{6}[[x_{4},x_{3}],[x_{2},x_{1}]]
\end{align*}
for some scalars $\beta_{1},\dots,\beta_{6}\in K$. Since the three last terms of the Lie polynomial $L$ can be written as a linear combination of product of two commutators we may write $f$ as
\begin{align*} 
f(x_{1},x_{2},x_{3},x_{4})= & \beta_{1}[[[x_{2},x_{1}],x_{3}],x_{4}]+\beta_{2}[[[x_{3},x_{1}],x_{2}],x_{4}]+\beta_{3}[[[x_{4},x_{1}],x_{2}],x_{3}]\\
+ & \alpha_{1}[x_{1},x_{2}][x_{3},x_{4}]+\alpha_{2}[x_{1},x_{3}][x_{2},x_{4}]+\alpha_{3}[x_{1},x_{4}][x_{2},x_{3}]\\ 
+ & \alpha_{4}[x_{2},x_{3}][x_{1},x_{4}]+\alpha_{5}[x_{2},x_{4}][x_{1},x_{3}]+\alpha_{6}[x_{3},x_{4}][x_{1},x_{2}]
\end{align*}

If some among the scalars $\beta_{1},\beta_{2},\beta_{3}$ is nonzero we claim that $sl_{n}(K)\subset f(M_{n}(K))$. Indeed, say $\beta_{1}\neq0$ without loss of generality. Take $x_{1}=x_{3}=x_{4}=D$ as a diagonal matrix with pairwise distinct diagonal entries. Since any matrix in $M_n(K)$ is the sum of a diagonal matrix and a matrix with zeros in the main diagonal, Lemma \ref{l5} implies that $f(D,x_{2},D,D)$ consists of all matrices with zeros in the main diagonal. On the other hand, Lemma \ref{lemasimiliar} states that traceless matrices are conjugate to those matrices with zero diagonal. Then by Proposition \ref{properties} (1) the claim is proved. The cases where $\beta_{2}\neq0$ and $\beta_{3}\neq0$ can be handled similarly. 

From now on we may assume $\beta_{1}=\beta_{2}=\beta_{3}=0$ and then we consider the two following cases.
\\

{\bf Case 1:} $\alpha_{1}=\alpha_{4}=\alpha_{6}=\alpha_{3}=-\alpha_{2}=-\alpha_{5}$. 
\\

The above assumptions on the coefficients of $f$ lead us to $f=\lambda St_{4}$ where $\lambda\in K\setminus\{0\}$. Using the identity $[uv,w]=[u,w]v+u[v,w]$ we have
\begin{eqnarray}\nonumber
St_{4}(A,A^{2},B,C)=[A,A^{2}][B,C]+[B,C][A,A^{2}]+[A^{2},B][A,C]+[A,C][A^{2},B]\\\nonumber 
-[A,B][A^{2},C]-[A^{2},C][A,B]\\\nonumber 
=[A^{2},B][A,C]+[A,C][A^{2},B]-[A,B][A^{2},C]-[A^{2},C][A,B]\\\nonumber 
=[A,[[A,B],[A,C]]]
\end{eqnarray}
Now is enough to apply Lemma \ref{l4} for the Jordan normal form of a traceless matrix.
\\

{\bf Case 2:} {\it at least one among the equalities} $\alpha_{1}=\alpha_{4}=\alpha_{6}=\alpha_{3}=-\alpha_{2}=-\alpha_{5}$ {\it does not hold.} 
\\

In this case one may check that there exist matrices $A,B,C \in M_{n}(K)$ such that at least one of the following is a nonzero matrix
\begin{eqnarray}\nonumber
f(A,A,B,C)=(\alpha_{2}+\alpha_{4})[A,B][A,C]+(\alpha_{3}+\alpha_{5})[A,C][A,B]\\\nonumber
f(A,B,A,C)=(\alpha_{1}-\alpha_{4})[A,B][A,C]+(\alpha_{6}-\alpha_{3})[A,C][A,B]\\\nonumber
f(A,B,C,A)=(-\alpha_{1}-\alpha_{5})[A,B][A,C]+(-\alpha_{2}-\alpha_{6})[A,C][A,B]\\\nonumber
f(B,A,A,C)=(-\alpha_{1}-\alpha_{2})[A,B][A,C]+(-\alpha_{5}-\alpha_{6})[A,C][A,B]\\\nonumber
f(B,A,C,A)=(-\alpha_{3}+\alpha_{1})[A,B][A,C]+(\alpha_{6}-\alpha_{4})[A,C][A,B]\\\nonumber
f(B,C,A,A)=(\alpha_{2}+\alpha_{3})[A,B][A,C]+(\alpha_{4}+\alpha_{5})[A,C][A,B]
\end{eqnarray}
Hence it is enough to study the image of the  polynomial 
\begin{eqnarray}\label{ex}
f=[x_{1},x_{2}][x_{1},x_{3}]+\lambda[x_{1},x_{3}][x_{1},x_{2}]
\end{eqnarray}
 on $M_{n}(K)$ where $\lambda\in K$. This will be done in the next two lemmas.
\end{proof}

\begin{lem}\label{l1}
 Let $K$ be an algebraically closed field of characteristic zero and let $n\geq3$. Then each $D\in sl_{n}(K)$ can be written as $D=[[A,B],[A,C]]$ for a suitable choice of matrices $A,B,C\in M_{n}(K)$.
\end{lem}

We note that Lemma \ref{l1} 
completely solve the case where $\lambda=-1$ in (\ref{ex}). The next lemma deals with the others values for $\lambda$.

\begin{lem}\label{l2}
Let $K$ be an algebraically closed field of characteristic zero, let  $n\geq 3$ and let $\lambda \in K\setminus\{-1\}$. Then each $D\in sl_{n}(K)$ can be written as $D=[A,B][A,C]+\lambda[A,C][A,B]$ for a suitable choice of matrices $A,B,C \in M_{n}(K)$.
\end{lem}

We first note that the proof of Lemma \ref{l2} presented in \cite{BW} is not correct, since it was claimed that given scalars $d_{1},\dots,d_{n}\in K$ satisfying $\sum_{i=1}^{n}d_{i}=0$ the following system of equations
\begin{eqnarray}\nonumber
\left\{\begin{array}{lcc}
(\lambda+1)y_{1}&=&d_{1}\\
(\lambda+1)y_{2}&=&d_{2}\\
&\vdots &\\
(\lambda+1)y_{n-1}&=&d_{n-1}\\
-(n-1)(\lambda+1)y_{n}&=&d_{n}
\end{array}\right.
\end{eqnarray} 
has a solution $y_{i}=b_{i}, i=1,\dots,n$ satisfying $\sum_{i=1}^{n}b_{i}=0$. However the existence of such solution would gives us $-n(\lambda+1)b_{n}=0$ by summing all equations above. Hence $b_{n}=0$ which implies $d_{n}=0$, a contradiction with the generality of the chosen $d_{n}$.

In next we will present a correction of the proof of Lemma \ref{l2}. We recall the following lemma from \cite{BW} which will be used in our proof.

\begin{lem}\label{l3}
Let $K$ be a field, let $a_{i,j}\in K$ such that $\sum_{i=1}^{n}a_{i,i}=0$ and let $A=\sum_{i=1}^{n-1}e_{i,i+1}\in M_{n}(K)$. Then there exists $B\in M_{n}(K)$ such that 
\begin{eqnarray}\nonumber
[A,B]=\sum_{i=1}^{n}a_{i,i}e_{i,i}+\sum_{i=1}^{n-1}a_{i,i+1}e_{i,i+1}.
\end{eqnarray} 
\end{lem}
 
\begin{proof}[Proof of Lemma \ref{l2}]
We start noting that we may assume $D$ is in  its Jordan normal form since $K$ is an algebraically closed field and the image of the polynomial 
\begin{eqnarray}\nonumber
f(x_{1},x_{2},x_{3})=[x_{1},x_{2}][x_{1},x_{3}]+\lambda[x_{1},x_{3}][x_{1},x_{2}]
\end{eqnarray}
is closed under conjugation by invertible elements of $M_{n}(K)$. 

So we write $D$ as
\begin{equation}\label{e1}
D=\sum_{i=1}^{n}d_{ii}e_{ii}+\sum_{i=1}^{n-1}d_{i,i+1}e_{i,i+1}
\end{equation}
where $d_{ii},d_{i,i+1}\in K$.
Take $A=\displaystyle\sum_{i=1}^{n-1}e_{i,i+1}$ and given any $a_{ii},b_{ii}\in K$, $i=1,\dots,n$ and $b_{i,i+1}\in K$, $i=1,\dots,n-1$ such that 
\begin{eqnarray}\nonumber
\sum_{i=1}^{n}a_{ii}=0=\sum_{i=1}^{n}b_{ii},
\end{eqnarray}
the Lemma \ref{l3} gives us the existence of matrices $B,C\in M_{n}(K)$ where $$[A,B]=\sum_{i=1}^{n}a_{ii}e_{ii} \ \mbox{and} \ [A,C]=\sum_{i=1}^{n}b_{ii}e_{ii}+\sum_{i=1}^{n-1}b_{i,i+1}e_{i,i+1}.$$

Therefore, 

\begin{equation} \label{e2}
[A,B][A,C]+\lambda[A,C][A,B]=(1+\lambda)\sum_{i=1}^{n}a_{ii}b_{ii}e_{ii}+\sum_{i=1}^{n-1}(a_{ii}+\lambda a_{i+1,i+1})b_{i,i+1}e_{i,i+1}
\end{equation}
and we are looking for a simultaneous solution of the two systems below given by comparing the equations (\ref{e1}) and (\ref{e2})

\begin{equation}\label{e3}
\left\{\begin{array}{c}
(1+\lambda)a_{11}b_{11}=d_{11}\\
\vdots\\
(1+\lambda)a_{nn}b_{nn}=d_{nn}\\
\end{array}\right.
\end{equation}

and 

\begin{equation}\label{e4}
\left\{\begin{array}{c}
(a_{11}+\lambda a_{22})b_{12}=d_{12}\\
\vdots\\
(a_{n-1,n-1}+\lambda a_{nn})b_{n-1,n}=d_{n-1,n}\\
\end{array}\right.
\end{equation}

jointly with the conditions $\displaystyle\sum_{i=1}^{n}a_{ii}=0=\sum_{i=1}^{n}b_{ii}$.

From now on we will divide our proof in the following three cases concerning about the number of Jordan blocks in the Jordan normal form of $D$.
\\

{\bf Case 1:} {\it $D$ has exactly one Jordan block.}
\\

In this case we must have $d_{ii}=0$ for all $i$, since $D$ is a traceless matrix and $char(K)=0$
.

If $\lambda\neq\displaystyle\frac{1}{n-1}$, one may check that is enough to choose $\displaystyle b_{11}=\cdots=b_{nn}=0, a_{11}=\cdots=a_{n-1,n-1}=1, a_{nn}=-(n-1), b_{12}=\frac{d_{12}}{1+\lambda}, b_{23}=\frac{d_{23}}{1+\lambda},\cdots,b_{n-2,n-1}=\frac{d_{n-2,n-1}}{1+\lambda}\ \mbox{and} \ b_{n-1,n}=\frac{d_{n-1,n}}{1+\lambda(1-n)}$.

If $\displaystyle\lambda=\frac{1}{n-1}\neq0$, then we choose  $\displaystyle b_{11}=\cdots=b_{nn}=0, a_{11}=\cdots=a_{n-2,n-2}=1,a_{n-1,n-1}=0, a_{nn}=-(n-2),b_{12}=\frac{d_{12}}{1+\lambda},\cdots,b_{n-3,n-2}=\frac{d_{n-3,n-2}}{1+\lambda},b_{n-2,n-1}=d_{n-2,n-1}\ \mbox{and}$ $\displaystyle b_{n-1,n}=\frac{d_{n-1,n}}{\lambda(2-n)}$, and we are done with the first case. 

Now we give a brief note about the notation before starting the next case. For simplicity we will write $a_{i}$ and $b_{i}$ instead of $a_{ii}$ and $b_{ii}$, respectively.

Assuming $a_{i}\neq0$, $i=1,\dots,n$, by the equations in (\ref{e3}) we have $(1+\lambda)a_{i}b_{i}=d_{ii}$ for $i=1,\dots,n$ and then $b_{i}=(1+\lambda)^{-1}a_{i}^{-1}d_{ii}$. Summing these equations for all $i$ we get that $\displaystyle(1+\lambda)^{-1}(\sum_{i=1}^{n}a_{i}^{-1}d_{ii})=0$, that is, $$\frac{d_{11}}{a_{1}}+\cdots+\frac{d_{nn}}{a_{n}}=0.$$
\\

{\bf Case 2:} {\it $D$ has exactly two Jordan blocks.}
\\

Suppose $D$ has one block of size $m_{1}$ and eigenvalue $d_{1}$ and another block of size $m_{2}$ and eigenvalue $d_{2}$. One between this two blocks must be of size at least two since $n\geq3$, and therefore we can take $m_{1}\geq 2$. Since $D$ has trace zero we have $d_{1}=0$ if and only if $d_{2}=0$, and then the previous case allow us to assume $d_{1}\neq0$. In this case we are looking for nonzero values for all $a_{i}$ such that $a_{i}+\lambda a_{i+1}$ is nonzero for all $i$. 
Note that this last condition can be used to compute the values of $b_{i,i+1}$ in (\ref{e4}) easily. 

Taking $a_{3}=\cdots=a_{n}=1$, we obtain

$$0=\frac{d_{1}}{a_{1}}+\cdots+\frac{d_{1}}{a_{m_{1}}}+\frac{d_{2}}{a_{m_{1}+1}}+\cdots+\frac{d_{2}}{a_{n}}= \frac{d_{1}}{a_{1}}+\frac{d_{1}}{-a_{1}-(n-2)}+(m_{1}-2)d_{1}-m_{1}d_{1.}$$ Therefore
$$\frac{1}{a_{1}}-\frac{1}{a_{1}+(n-2)}-2=0, \ \mbox{i.e.}, \ a_{1}+(n-2)-a_{1}-2a_{1}^2-2(n-2)a_{1}=0,$$ which lead us to the following equation
\begin{eqnarray}\label{e5}
2a_{1}^2+2(n-2)a_{1}-(n-2)=0
\end{eqnarray}
Since $n-2\neq0$, then $a_{1}\neq0$. We have also $a_{1}\neq-(n-2)$, otherwise $$2(n-2)^2+2(n-2)(-(n-2))-(n-2)=0, \ \mbox{that is}, \ n=2.$$

We note that the roots of the equation (\ref{e5}) in $a_{1}$ are $$a_{1}=\frac{-(n-2)\pm \sqrt{n(n-2)}}{2}$$
and so $$a_{2}=\frac{-(n-2)\mp\sqrt{n(n-2)}}{2},$$ provided that $a_{2}=-a_{1}-(n-2)$.

In computing $b_{i,i+1}$, the variables $b_{12}$ and $b_{23}$ depend on $a_{1}$ and $a_{2}$ by the equations $$(a_{1}+\lambda a_{2})b_{12}= d_{12} \ \mbox{and} \ (a_{2}+\lambda)b_{23}=d_{23}.$$

Denote $$\bar{a}_{1}=\displaystyle\frac{-(n-2)+ \sqrt{n(n-2)}}{2} \ \mbox{and} \ \displaystyle \bar{a}_{2}=\frac{-(n-2)-\sqrt{n(n-2)}}{2}.$$ 
For $\lambda\neq -\bar{a}_{1}/\bar{a}_{2}$ and $\lambda\neq -\bar{a}_{2}$, we take $a_{1}=\bar{a}_{1}$ and $a_{2}=\bar{a}_{2}$. For $\lambda\neq-\bar{a}_{1}/\bar{a}_{2}$ and $\lambda=-\bar{a}_{2}$, we take $a_{1}=\bar{a}_{2}$ and $a_{2}=\bar{a}_{1}$. So in the first equation we will have $\bar{a}_{2}-\bar{a}_{2}\bar{a}_{1}$ which is non-zero since $\bar{a}_{1}\neq 1$ and in the second equation we will have $\bar{a}_{1}-\bar{a}_{2}$ that is also non-zero. The last case is $\lambda=-\bar{a}_{1}/\bar{a}_{2}$. Again we take $a_{1}=\bar{a}_{2}$ and $a_{2}=\bar{a}_{1}$. Hence $\bar{a}_{2}+\lambda \bar{a}_{1}\neq0$, otherwise  $\bar{a}_{1}/\bar{a}_{2}=\bar{a}_{2}/\bar{a}_{1}$, i.e., $\bar{a}_{1}^2=\bar{a}_{2}^2$ which is an absurd. We also have $\bar{a}_{1}-\bar{a}_{1}/\bar{a}_{2}\neq0$, otherwise $\bar{a}_{2}=1$, another absurd.

Since all $\lambda \in K\setminus \{-1\}$ was considered above we finished the proof of the second case.
\\

{\bf Case 3:} {\it $D$ has $k\geq 3$ Jordan blocks.}
\\

Now suppose that $D$ is in the Jordan normal form with $k\geq3$ blocks of size $m_{k}$ each.
For the matrix $[A,B]=\displaystyle\sum_{i=1}^{n}a_{ii}e_{ii}$, we will consider the same block division that occurs in $D$ and in a same block we take all $a_{ii}$ equal to each other. For every $j\in\{1,\dots,k\}$, we will denote the element on the principal diagonal of the $j$-th block of $[A,B]$ by $a_{j}$.

We assume that $a_{j}\neq0$ for all $j\in\{1,\dots,k\}$. Since $(1+\lambda)a_{ii}b_{ii}=d_{ii}$, then $b_{ii}=(1+\lambda)^{-1}a_{ii}^{-1}d_{ii}$ and so $\displaystyle\sum_{i=1}^{n}a_{ii}^{-1}d_{ii}=0.$ Therefore,

$$0=\frac{m_{1}d_{1}}{a_{1}}+\cdots+\frac{m_{k-1}d_{k-1}}{a_{k-1}}+\frac{m_{k}d_{k}}{a_{k}}$$
$$=\frac{m_{1}d_{1}}{a_{1}}+\cdots+\frac{m_{k-1}d_{k-1}}{a_{k-1}}+ \frac{-\displaystyle\sum_{j=1}^{k-1}m_{j}d_{j}}{-\displaystyle\sum_{j=1}^{k-1}\frac{m_{j}}{m_{k}}a_{j}}$$
$$=\frac{m_{1}d_{1}}{a_{1}}+\cdots+\frac{m_{k-1}d_{k-1}}{a_{k-1}}+\frac{m_{k} \displaystyle\sum_{j=1}^{k-1}m_{j}d_{j}}{\displaystyle\sum_{j=1}^{k-1}m_{j}a_{j}}.$$

Taking $a_{1}=\cdots=a_{k-2}=1$, we obtain 

$$0=m_{1}d_{1}+\cdots+m_{k-2}d_{k-2}+\frac{m_{k-1}d_{k-1}}{a_{k-1}}+\frac{m_{k}\displaystyle \sum_{j=1}^{k-1}m_{j}d_{j}}{\displaystyle\sum_{j=1}^{k-2}m_{j}+m_{k-1}a_{k-1}},$$ and hence

$$0=a_{k-1}\bigg(\displaystyle\sum_{j=1}^{k-2}m_{j}+m_{k-1}a_{k-1}\bigg)\bigg(\sum_{j=1}^{k-2}m_{j}d_{j}\bigg) +m_{k-1}d_{k-1}\bigg(\sum_{j=1}^{k-2}m_{j}+m_{k-1}a_{k-1}\bigg)$$$$+a_{k-1}\bigg(m_{k}\sum_{j=1}^{k-1}m_{j}d_{j}\bigg)$$

$$ =m_{k-1}\bigg(\sum_{j=1}^{k-2}m_{j}d_{j}\bigg)a_{k-1}^2+\bigg(\bigg(\sum_{j=1}^{k-2}m_{j}\bigg)\bigg(\sum_{j=1}^{k-2}m_{j}d_{j}\bigg)+m_{k-1}^2d_{k-1}+m_{k}\sum_{j=1}^{k-1}m_{j}d_{j}\bigg)a_{k-1}$$ $$ + m_{k-1}d_{k-1}\sum_{j=1}^{k-2}m_{j}.$$ 

Denoting $d=\displaystyle\sum_{j=1}^{k-2}m_{j}d_{j}$ we obtain the follow quadratic equation in $a_{k-1}$:

\begin{equation} \label{e6}
m_{k-1}da_{k-1}^2 + \bigg(d\sum_{j=1}^{k-2}m_{j}+m_{k}d+m_{k-1}d_{k-1}(m_{k-1}+m_{k})\bigg)a_{k-1} +m_{k-1}d_{k-1}\sum_{j=1}^{k-2}m_{j}=0
\end{equation}

We want that $a_{k-1}\neq0$ e $a_{k}\neq0$, and since $a_{k}=\displaystyle \frac{-\sum_{j=1}^{k-2}m_{j}-m_{k-1}a_{k-1}}{m_{k}}$, then we are looking for non-zero solutions of (\ref{e6}) and both different from $\displaystyle -\frac{\sum_{j=1}^{k-2}m_{j}}{m_{k-1}}$. We divide this task in the three following subcases:
\\

{\bf Subcase 1:} $m_{k-1}d\neq0$ and $m_{k-1}d_{k-1}\displaystyle\sum_{j=1}^{k-2}m_{j}\neq0$.
\\

We have already two non-zero solutions in this case. Now we prove that at least one of them is different from $\displaystyle -\frac{\sum_{j=1}^{k-2}m_{j}}{m_{k-1}}$. Suppose, by contradiction, that the equation (\ref{e6}) has two repeated roots equal to $-\displaystyle\sum_{j=1}^{k-2}\frac{m_{j}}{m_{k-1}}$, i.e.,

\begin{equation}\label{e7}
\bigg(a_{k-1}+\sum_{j=1}^{k-2}\frac{m_{j}}{m_{k-1}}\bigg)^2=0.
\end{equation}

Hence $a_{k-1}=-\displaystyle\sum_{j=1}^{k-2}\frac{m_{j}}{m_{k-1}}$ and for other hand using the well known formula for the sum of the roots of a quadratic equation we also have $$a_{k-1}=-\displaystyle\frac{d\sum_{j=1}^{k-2}m_{j}+m_{k}d+m_{k-1}d_{k-1}(m_{k-1} +m_{k})}{2m_{k-1}d}$$ $$=-\sum_{j=1}^{k-2}\frac{m_{j}}{2m_{k-1}}-\frac{m_{k}}{2m_{k-1}}- \frac{d_{k-1}}{2d}(m_{k-1}+m_{k}).$$
Therefore, $$\frac{d_{k-1}}{2d}(m_{k-1}+m_{k})=-\frac{m_{k}}{2m_{k-1}}+\frac{1}{2}\sum_{j=1}^{k-2} \frac{m_{j}}{m_{k-1}},\mbox{i.e.,}$$ 

\begin{equation}\label{e8}
\frac{d_{k-1}}{d}=\frac{\sum_{j=1}^{k-2}m_{j}-m_{k}}{m_{k-1}(m_{k-1}+m_{k})}
\end{equation}

Dividing the equation (\ref{e6}) by $m_{k-1}d$ and comparing with the equation (\ref{e7}), we obtain

$$\left\{\begin{array}{l}
\displaystyle\sum_{j=1}^{k-2}\frac{m_{j}}{m_{k-1}}+\frac{m_{k}}{m_{k-1}}+\frac{d_{k-1}}{d}(m_{k-1}+m_{k})= 2\sum_{j=1}^{k-2}\frac{m_{j}}{m_{k-1}}\\
\displaystyle\frac{d_{k-1}}{d}\sum_{j=1}^{k-2}m_{j}=\frac{1}{m_{k-1}^2}\bigg(\sum_{j=1}^{k-2}m_{j}\bigg)^2.
\end{array}\right.$$

Since $\displaystyle\sum_{j=1}^{k-2}m_{j}\neq0$, then $\displaystyle\frac{d_{k-1}}{d}=\frac{1}{m_{k-1}^2}\bigg(\sum_{j=1}^{k-2}m_{j}\bigg)$ and so we get

$$2m_{k-1}\frac{d_{k-1}}{d}=\displaystyle\sum_{j=1}^{k-2}\frac{m_{j}}{m_{k-1}}+\frac{m_{k}}{m_{k-1}}+\frac{d_{k-1}}{d}(m_{k-1}+m_{k})$$

$$=\displaystyle\frac{n-m_{k-1}}{m_{k-1}}+\frac{d_{k-1}}{d}(m_{k-1}+m_{k}),$$ which implies in

$$\displaystyle\frac{d_{k-1}}{d}(m_{k-1}-m_{k})=\frac{n-m_{k-1}}{m_{k-1}}.$$

By the equation (\ref{e8}), we have

$$\displaystyle\frac{n-m_{k-1}}{m_{k-1}}=\frac{\sum_{j=1}^{k-2}m_{j}-m_{k}}{m_{k-1}(m_{k-1}+m_{k})} (m_{k-1}-m_{k}), \ \mbox{that is},$$

$$\displaystyle n-m_{k-1}=\frac{n-m_{k-1}-2m_{k}}{m_{k-1}+m_{k}}(m_{k-1}-m_{k}).$$

Therefore we have
\begin{eqnarray}\nonumber
(n-m_{k-1})(m_{k-1}+m_{k})=(n-m_{k-1}-2m_{k})(m_{k-1}-m_{k}),
\end{eqnarray}
and opening the brackets we obtain
\begin{eqnarray}\nonumber
nm_{k-1}+nm_{k}-m_{k-1}^2-m_{k-1}m_{k}\\ \nonumber 
=nm_{k-1}-nm_{k}-m_{k-1}^2+m_{k-1}m_{k}-2m_{k-1}m_{k}+2m_{k}^2
\end{eqnarray}
which implies in
$$m_{k}(n-m_{k})=0,\ \mbox{which means} \ m_{k}=0 \ \mbox{ou} \ n=m_{k},$$
and in the both cases we get a contradiction.

We conclude that there exists a non-zero root of (\ref{e6}) different from $\displaystyle -\sum_{j=1}^{k-2}\frac{m_{j}}{m_{k-1}}$.

{\bf Subcase 2:} $m_{k-1}d\neq0$ and $\displaystyle m_{k-1}d_{k-1}\sum_{j=1}^{k-2}m_{j}=0$.

Since $m_{k-1}$ and $\displaystyle\sum_{j=1}^{k-2}m_{j}$ are non-zero, then $d_{k-1}=0$. So the equation (\ref{e6}) can be rewritten as

$$m_{k-1}da_{k-1}^2+d\bigg(\sum_{j=1}^{k-2}m_{j}+m_{k}\bigg)a_{k-1}=0.$$

A solution of the equation above is $\displaystyle a_{k-1}=-\frac{\sum_{j=1}^{k-2}m_{j}+m_{k}}{m_{k-1}}\neq0$ which is also different from $\displaystyle -\sum_{j=1}^{k-2}\frac{m_{j}}{m_{k-1}}$, otherwise we would have $m_{k}=0$, a contradiction. 
\\

{\bf Subcase 3:} $m_{k-1}d=0$.

In this last case we have $d=0$ and so the equation (\ref{e6}) turns into 
$$\displaystyle d_{k-1}\bigg((m_{k-1}+m_{k})a_{k-1}+\sum_{j=1}^{k-2}m_{j}\bigg)=0.$$

If $d_{k-1}=0$, then any element of $K$ is solution and therefore we choose the appropriate one.

If $d_{k-1}\neq0$, then $\displaystyle a_{k-1}=-\frac{\sum_{j=1}^{k-2}m_{j}}{m_{k-1}+m_{k}}$. Provided that $k\geq 3$ we have $a_{k-1}\neq0$ and since $m_{k}\neq0$ we obtain $a_{k-1}\neq \displaystyle -\sum_{j=1}^{k-2}\frac{m_{j}}{m_{k-1}}$.

Now is enough to determine the values for $b_{i,i+1}$ in the system (\ref{e4}).

In the matrix $D$, from the end of the block that contains the element $d_{ii}$ to the beginning of the one containing $d_{i+1,i+1}$ we have $d_{i,i+1}=0$ and then we can take $b_{i,i+1}=0$. For the other elements above the principal diagonal we have $(a_{ll}+\lambda a_{l+1,l+1})b_{l,l+1}=d_{l,l+1}$ and $a_{ll}=a_{l+1,l+1}$. So we can take $b_{l,l+1}=a_{ll}^{-1}(1+\lambda)^{-1}d_{l,l+1}$. 

This finish the proof of the third and last case and therefore we conclude the proof of the lemma.
\end{proof}

\section*{Funding}

P. S. Fagundes was supported by São Paulo Research Foundation (FAPESP), grants \#2016/09496-7 and 	
\#2019/16994-1. T. C. de Mello was supported by São Paulo Research Foundation (FAPESP), grant \#2018/23690-6. P. H. S. dos Santos was financed in part by the Coordenação de Aperfeiçoamento de Pessoal de Nível Superior - Brasil (CAPES) - Finance Code 001.

\end{document}